\documentclass[11pt]{article}
\usepackage{amsfonts,amsmath,amssymb,amsthm}
\usepackage{hyperref}
\usepackage{kbordermatrix}
\newtheorem{theorem}{Theorem}[section]

\newtheorem{lemma}[theorem]{Lemma}
\newtheorem{claim}[theorem]{Claim}

\newcommand{\rank}{\mbox{rank }}

\newcommand{\R}{{\mathbb R}}
\newcommand{\real}{{\mathbb R}}
\newcommand{\rat}{{\mathbb Q}}
\newcommand{\sm}{{\setminus}}

\newcommand{\cT}{{\mathcal T}}

\title{Vertex Splitting, Coincident Realisations and Global Rigidity of Braced Triangulations}
\date{\today}
\author{James Cruickshank\thanks{School of Mathematics, Statistics and Applied Mathematics, National
University of Ireland, Galway, Ireland. E-mail: james.cruickshank@nuigalway.ie}, Bill Jackson\thanks{School of Mathematical Sciences, Queen Mary
University of London, Mile End Road, London E1 4NS, United Kingdom.
E-mail: b.jackson@qmul.ac.uk} and Shin-ichi Tanigawa\thanks{Department of Mathematical Informatics, Graduate School of Information Science and Technology, University of Tokyo, 7-3-1 Hongo, Bunkyo-ku, 113-8656,  Tokyo Japan. Email: {tanigawa@mist.i.u-tokyo.ac.jp}}}

\usepackage{color}

\begin{document} \maketitle

\begin{abstract}
We give a relatively short graph theoretic proof of a result of Jord\'an and Tanigawa that a 4-connected graph which has a spanning plane triangulation as a proper subgraph is generically globally rigid in $\real^3$.  Our proof is based on a new  sufficient condition for the so called vertex splitting operation to preserve generic global rigidity in $\real^d$. 

\end{abstract}



\section{Introduction}
We consider the problem of determining when a configuration consisting of a finite set of points in $d$-dimensional Euclidean space $\real^d$ is uniquely defined up to congruence by a given set of constraints which fix the Euclidean distance between certain pairs of points. 

More formally, we define a {\em $d$-dimensional framework} to be a pair $(G,p)$ where $G=(V,E)$ is a graph and $p:V\to \R^d$ 
is a point configuration. 
We will also refer to the framework $(G,p)$ and the configuration $p$ as  {\em realisations} of $G$ in $\R^d$. The length of an edge of $G$ in $(G,p)$ is given by the Euclidean distance between its endvertices. Two realisations $(G,p)$ and $(G,q)$ of $G$ in $\R^d$ are {\em congruent} if $(G,p)$ can be obtained from $(G,q)$ by an isometry of $\R^d$ i.e. a combination of translations, rotations and reflections. 
The framework $(G,p)$ is {\em globally rigid} if every framework which has the same edge lengths as $(G,p)$ is congruent to $(G,p)$. It is {\em rigid} if every continuous motion of the vertices of $(G,p)$ in $\R^d$ which preserves the edge lengths results in a framework which is congruent to $(G,p)$.  It is {\em infinitesimally rigid} if it satisfies the stronger property that every infinitesimal motion of the points which preserves the edge lengths is induced by an infinitesimal isometry of $\R^d$ i.e. a combination of infinitesimal translations and rotations. (A formal definition of infinitesimal rigidity will be given in Section \ref{sec:inf}.)

Saxe \cite{S} showed that it is NP-hard to determine when a realisation of $G$ in $\R^d$ is globally rigid for all $d\geq 1$. Abbot \cite{A} showed that the same holds for rigidity for all $d\geq 2$. (It is straightforward to show that a 1-dimensional framework is rigid if and only if its underlying graph is connected.) These decision problems become more tractable, however,  if we restrict our attention to `generic realisations'. A configuration $p$ or framework $(G,p)$ is said to be {\em generic} if 
the set of coordinates of the points $p(v)$, $v\in V(G)$, is algebraically independent over $\rat$. Gluck \cite{G} showed that rigidity and infinitesimal rigidity are equivalent properties of  generic frameworks and depend only on the underlying graph. This enables us to define a graph $G$ as being {\em rigid in $\R^d$} if some (or equivalently every) generic realisation of $G$ in $\R^d$ is rigid (or equivalently infinitesimally rigid). Analogous, but much deeper, results of Connelly \cite{Cglob}
and Gortler, Healy and Thurston \cite{GHT} imply that the global rigidity of a generic framework depends only on its underlying graph and allows us to 
define a graph $G$ as being {\em globally rigid in $\R^d$} if some (or equivalently every) generic realisation of $G$ in $\R^d$ is globally rigid.
It is 
known that a graph is globally rigid in $\R$ if and only if it is 2-connected. Characterisations of graphs which are rigid or globally rigid in $\R^2$ are given in \cite{L,PG} and \cite{JJ}, respectively.
It is a major open problem in distance geometry to characterise rigid or globally rigid graphs when $d\geq 3$, although many partial results exist for particular families of graphs. We refer the reader to \cite{SW, JW} for  recent survey articles on framework rigidity.

%

Gluck~\cite{G} showed that every plane triangulation, i.e. maximal planar graph on at least three vertices, is rigid in $\R^3$. Plane triangulations other than $K_3$ and $K_4$ do not have enough edges to be globally rigid but   
a recent result of Jord\'an and Tanigawa \cite{JT} characterises when {\em braced plane triangulations}, i.e. graphs constructed from plane triangulations by adding additional edges called {\em braces}, are globally rigid in $\real^3$.

\begin{theorem}\label{JKT} Suppose that $G$ is a braced plane triangulation with at least five vertices.  Then $G$ is globally rigid in $\R^3$ if and only if $G$ is $4$-connected and has at least one brace.
\end{theorem}

 We will give a relatively short proof of this result.

\subsection{Vertex splitting} 
Our proof of Theorem~\ref{JKT} is inductive and is based on  (the 3-dimensional version of) Theorem \ref{vsplitc} below, which verifies a special case of the so-called vertex splitting conjecture. Before stating this theorem we need to introduce some new terminology.

%
%
%
%
%
%
%

Given a graph $G=(V,E)$ and $v\in V$ with neighbour set  $N_G(v)$, the {\em $d$-dimensional vertex splitting operation}  constructs a new graph $G'$ by choosing  pairwise disjoint sets  $U_1,U_2,U_3$ with $U_1\cup U_2\cup U_3=N_G(v)$ and $|U_2|=d-1$, deleting all edges from $v$ to $U_3$, and then adding a new vertex $v'$ and $|U_3|+d$ new edges from $v'$ to each vertex in $U_2\cup U_3\cup \{v\}$.
Whiteley \cite{Wsplit} showed that the vertex splitting operation preserves rigidity in $\R^d$ and  
conjectured that it also preserves global rigidity in $\R^d$ whenever $v$ and $v'$ both have degree at least $d+1$ in $G'$, see \cite{CW,JW}.
We will verify a special case of this conjecture.

\begin{theorem}\label{vsplitc} Let $G=(V,E)$ be a graph which is 
globally rigid in $\R^d$ and $v\in V$. Suppose that $G'$ is obtained
from $G$ by a $d$-dimensional vertex splitting operation which splits $v$ into two vertices $v$ and $v'$,  and that $G'$ can be realised as an
infinitesimally rigid framework in $\R^d$ in which $v$ and $v'$ are coincident. Then
$G'$ is generically globally rigid in $\R^d$.
\end{theorem}

\noindent
The assertion that Theorem \ref{vsplitc} is a special case of Whiteley's conjecture follows from Lemma \ref{lem:coin_nec} below (which implies that, if $G'$ can be realised as a $vv'$-coincident
infinitesimally rigid framework in $\R^d$, then $G'-vv'$ is rigid so $v$ and $v'$ must both have degree at least $d+1$ in $G'$).

Theorem \ref{vsplitc}  can be deduced from  the results of Connelly\cite{C} using the stress matrix characterisation of global rigidity in \cite{GHT}. We give a  shorter direct proof in Section \ref{sec:vsplit}.
Theorem \ref{vsplitc}
has already been used by Jord\'an, Kiraly and Tanigawa in \cite{JKT1} to repair a gap in the proof of their characterisation of generic global rigidity 
for body-hinge frameworks given in \cite{JKT}.\footnote{They mistakenly state   in \cite{JKT} that a stronger form of  Theorem \ref{vsplitc} (in which the hypothesis  that $G'$ can be realised as a $vv'$-coincident
infinitesimally rigid framework in $\R^d$ is replaced by the hypothesis  that $G'-vv'$ is rigid) is implied by  \cite{C}.} An analogous result to Theorem \ref{vsplitc} was used in \cite{JN1,JN2} to obtain a characterisation of generic global rigidity for cylindrical frameworks.

\subsection{Coincident rigidity}
In order to use Theorem~\ref{vsplitc}, we need to have sufficient conditions for a graph to have an  infinitesimally rigid realisation  with two coincident points.
The analysis of such graphs was initiated by Fekete, Jord\'an and
Kaszanitzky \cite{FJK}, who gave a complete combinatorial characterization in the two dimensional case.

Given two vertices $u,v$ in a graph $G=(V,E)$ we use $G-uv$ to denote the graph obtained from $G$ by deleting the edge $uv$ if it exists in $G$ (and putting $G-uv=G$ if it does not exist). We also use $G/uv$ to denote the graph obtained from $G$ by replacing $u$ and $v$ by a single vertex $w$ which is adjacent to every neighbour of $u$ and $v$ in $V\sm \{u,v\}$. We say that $G$ is {\em $uv$-coincident rigid} in $\R^d$ if $G$ can be realised  as an infinitesimally rigid framework $(G,p)$ in $\R^d$ with $p(u)=p(v)$. 
Fekete, Jord\'an and
Kaszanitzky \cite{FJK} showed that $G$ is $uv$-coincident rigid in $\R^2$ if and only if $G-uv$ and $G/uv$ are both 
rigid in $\R^2$. Their proof of necessity extends immediately to $\R^d$.

\begin{lemma}\label{lem:coin_nec}
Suppose $u,v$ are two vertices of a graph $G$. If $G$ is $uv$-coincident rigid  in $\R^d$  then $G-uv$ and $G/uv$ are both 
rigid in $\R^d$.
\end{lemma}

The converse direction does not hold in $\mathbb{R}^3$ (see \cite[Section 5.2]{GJ}), and it is an open problem to characterize  $uv$-coincident rigidity in terms of rigidity in $\mathbb{R}^d$.

In order to link Theorems \ref{JKT} and \ref{vsplitc}, 
we will obtain sufficient conditions for  the $uv$-coincident rigidity of braced plane triangulations in $\mathbb{R}^3$.
Lemma \ref{lem:coin_nec} implies that no plane triangulation $T$ can be $uv$-coincident rigid in $\mathbb{R}^3$ for two adjacent vertices $u,v$ since $T-uv$ has too few edges to be rigid. 
Our third main result   gives a sufficient condition  for $T$  to become  $uv$-coincident rigid after the addition of at least one brace. 


\begin{theorem}\label{thm:coincident1}
Let $G$ be a $4$-connected braced plane triangulation which is obtained from a plane triangulation $T$ by adding at least one brace. Suppose that  $e=uv$ is an edge of $T$ which does not belong to any separating $3$-cycle of $T$. 
Then $G$ is $uv$-coincident rigid  in $\mathbb{R}^3$
\end{theorem}

In the forthcoming paper~\cite{CJT}, we shall verify a conjecture of Connelly concerning  the global rigidity of triangulated surfaces in $\mathbb{R}^3$. 
Theorems~\ref{vsplitc} and~\ref{thm:coincident1} will be key ingredients in our proof.

\section{Infinitesimally rigid realisations}\label{sec:inf}

We can determine whether a given $d$-dimensional framework $(G,p)$ is infinitesimally rigid by calculating the rank of its {\em rigidity matrix}. This is the matrix of size $|E|\times d|V|$ in which each row is indexed by an edge, sets of $d$ consecutive columns are indexed by the vertices, and the row indexed by the  edge $e=uv$ has the form:
\[
\kbordermatrix{
  & &  u & & v & \\
 e=uv & 0 \dots 0 & p(u)-p(v) & 0\dots 0 & p(v)-p(u) & 0\dots 0
}.
\]
The space of {\em infinitesimal motions}  of $(G,p)$ is given by the right kernel of $R(G,p)$.   The framework  $(G,p)$ is  {\em infinitesimally rigid} if $\rank R(G,p)=d|V|-{d+1\choose 2}$ when $|V|\geq d$ and $\rank R(G,p)={|V|\choose 2}$ when $|V|<d$, since this will imply that every vector in $\ker R(G,p)$ is an infinitesimal isometry of $\R^d$.
Since the rank of $R(G,p)$ will be maximised whenever $(G,p)$ is generic, 
the infinitesimal  rigidity of a generic framework $(G,p)$ in $\R^d$ depends only on the underlying graph $G$.
We say that  $(G,p)$ is {\em independent} if $R(G,p)$ is row independent, and is {\em minimally rigid} if it is both infinitesimally rigid and independent. 

The fact that the entries in the rigidity matrix of $(G,p)$ are linear in the coordinates of $p$ implies that, if $G$ is $uv$-coincident rigid and  $(G,p)$ is a $uv$-coincident  realisation of $G$ in which  $p|_{V-v}$ is generic, then $(G,p)$ will be  infinitesimally rigid. We will refer to such a realisation as a {\em generic $uv$-coincident realisation} of $G$.

In order to apply Theorem \ref{vsplitc}, we need to construct an infinitesimally rigid realisation of a graph in which two given vertices are coincident.  We will use  the following results on infinitesimal rigidity to do this.  

We first give a precise statement of Whiteley's above mentioned  vertex splitting theorem \cite[Corollary 11]{Wsplit}.
%
%
%
%
%
%

\begin{lemma}\label{lem:split} 
Suppose that $(G,p)$ is an infinitesimally rigid framework  in $\R^d$ and that $G'$ is obtained from $G$ by a vertex splitting operation which splits a vertex $v$ of $G$ into two vertices $v',v''$. Suppose further that, for some $X\subseteq  N_{G'}(v')\cap N_{G'}(v'')$ with $|X|=d-1$, the points in $\{p(w)\,:\,w\in X+v\}$ are affinely independent in $\R^d$. 
Then $(G',p')$ is infinitesimally rigid for some $p'$ with $p'(v')=p(v)$ and $p'(w)=p(w)$ for all $w\in V(G)\setminus \{v\}$.
\end{lemma}



Our next two lemmas concern the so called 1-extension and gluing operations. 
The {\em $d$-dimensional $1$-extension operation} on a graph $G$ constructs  a new graph $G'$ by deleting an edge $uv$ and then adding a new vertex $w$ with $u,v\in N_{G'}(w)$ and $|N_{G'}(w)|=d+1$. The {\em $d$-dimensional gluing operation} constructs a new graph by taking the union of two graphs with at least $d$ vertices in common.
These lemmas can be proved using standard techniques: for example, Lemma \ref{lem:1ext} follows from the first part of  the proof of \cite[Theorem 11.1.7]{Wlong}, and the proof of  \cite[Lemma 3.1.4]{Wlong} easily extends to give Lemma \ref{lem:glue}.

\begin{lemma}\label{lem:1ext} 
Suppose that $(G,p)$ is an infinitesimally rigid framework  in $\R^d$ and that $G'$ is obtained from $G$ by a $d$-dimensional $1$-extension operation which  adds a new vertex $w$. Suppose further that the points $\{p(x)\,:\,x\in N_{G'}(w)\}$ are in general position in $\R^d$.
Then there is an extension $p':V(G')\rightarrow \mathbb{R}^d$ of $p$ such that 
$(G',p')$ is infinitesimally rigid.
\end{lemma}


\begin{lemma}\label{lem:glue} 
Let $(G,p)$ be a framework  in $\R^d$, $G_1$ and $G_2$ be subgraphs of $G$ with $G=G_1\cup G_2$,  and $U\subseteq V(G_1)\cap V(G_2)$ with $|U|=d$. Suppose that  $(G_1,p|_{G_1})$ and  $(G_2,p|_{G_2})$ are infinitesimally rigid and $p(U)$ is in general position in $\R^d$. Then  
$(G,p)$ is infinitesimally rigid.
\end{lemma}

To prove Theorem \ref{JKT}, we also need the following combination of Lemmas \ref{lem:1ext} and \ref{lem:glue}.

\begin{lemma}\label{lem:elem} Let $G_1,G_2$ be graphs which are rigid in $\R^d$ and satisfy $|(V(G_1)\cap V(G_2))|\geq d$,
$x\in V(G_1)\sm V(G_2)$, $y\in V(G_2)\sm V(G_1)$, $z\in V(G_1)\cap V(G_2)$ and $xz\in E(G_1)$.  Put $G=(G_1\cup G_2)-xz+xy$. Suppose that $(G_1,p_1)$ is an infinitesimally rigid realisation of $G_1$ and that $p_1$ is generic on $(V(G_1)\cap V(G_2))\cup\{x\}$. Then $(G,p)$  is infinitesimally rigid for some $p$ with $p(v)=p_1(v)$ for all $v\in V(G_1)$.
\end{lemma}
\begin{proof} Let $(G'_1,p_1')$ be obtained from $(G_1-xz,p_1)$ by adding the vertex $y$ at a point $p_1'(y)$ whose coordinates are algebraically independent over 
the field obtained by extending  $\rat$ by the coordintes of $p_1$, and then adding an edge from $y$ to $x$ and all vertices in $V(G_1)\cap V(G_2)$. Then $(G_1',p_1')$ is infinitesimally rigid by Lemma \ref{lem:1ext} (since it can be obtained from $(G_1,p_1)$ by a 1-extension
and a possibly empty sequence of edge additions). 
 Lemma \ref{lem:glue} now implies that $(G_1'\cup G_2,p)$ is infinitesimally rigid for any generic extension $p$ of $p_1'$. We can now deduce that that $(G,p)$ is infinitesimally rigid (since $(G_2,p|_{G_2})$ is infinitesimally rigid,  we can delete any edges of $G_1'$ from $y$ to  $V(G_1)\cap V(G_2)$ which do not belong to $G_2$ without destroying the infinitesimal rigidity of $(G_1'\cup G_2,p)$).  
\end{proof}

\section{Global rigidity and vertex splitting}\label{sec:vsplit}
In this section we  prove Theorem \ref{vsplitc}.
We need the following result of Connelly and Whiteley \cite[Theorem 13]{CW}, which shows that global rigidity is a stable property for infinitesimally rigid frameworks, to prove Theorem \ref{vsplitc}.

\begin{lemma}\label{lem:stable} 
Suppose that $(G,p)$ is an infinitesimally rigid, globally rigid framework on $n$ vertices in $\real^d$. Then there exists an open neighbourhood $N_p$ of $p$ in $\real^{dn}$ such that $(G,q)$ is infinitesimally rigid and globally rigid for all $q\in N_p$.
\end{lemma}

Note that, although the definition of a framework $(G,p)$ in \cite{CW} requires $p(u)\neq p(v)$ for all $uv\in E(G)$, this simplifying assumption is not needed in the proof of \cite[Theorem 13]{CW}.

\paragraph{Proof of Theorem \ref{vsplitc}.}
Let $(G,p)$ be a generic realisation of $G$ in $\real^d$ and let $(G',p')$ be the $vv'$-coincident realisation of $G'$ obtained by putting $p'(u)=p(u)$ for all $u\in V$ and $p'(v')=p(v)$. The genericity of $p$ implies that the rank of the rigidity matrix of any $vv'$-coincident realisation of $G'$ will be maximised at $(G',p')$ and hence $(G',p')$ is infinitesimally rigid. The genericity of $p$ also implies that $(G,p)$ is globally rigid, and this in turn implies that $(G',p')$ is globally rigid. We can now use Lemma \ref{lem:stable} to deduce that $(G',q)$ is globally rigid for any generic $q$ sufficiently close to $p'$. Hence $G'$ is globally rigid.
\qed

\vspace{1em}

We will be exclusively concerned with 3-dimensional frameworks in the remainder of the paper so will henceforth suppress reference to the ambient space $\R^3$  and say, for example, that a graph is rigid to mean it is rigid in $\R^3$.

\section{Coincident rigidity of plane triangulations}\label{sec:contract}
A graph $T$ is a {\em plane triangulation} if it has a 2-cell embedding in the plane in which every face has three edges on its boundary. 
The infinitesimal rigidity of realisations of plane triangulations in $\R^3$ is one of the fundamental topics in graph rigidity, see, e.g.,~\cite{W}. 
 As a warm up to the analysis of $uv$-coincident rigidity of braced plane triangulations, we shall analyze the $uv$-coincident rigidity of plane triangulations.
As remarked in the introduction, no plane triangulation $T$ can be $uv$-coincident rigid in $\mathbb{R}^3$ for two adjacent vertices $u, v$ in $T$. 
We will give a sufficient condition for the $uv$-coincident rigidity of $T$ when $u,v$ are not adjacent. This result will be  used to obtain our characterisation of globally rigid triangulated surfaces in \cite{CJT}.
 
 We will need the following notation and elementary results for (a particular embedding of) a plane triangulation $T$. Every cycle $C$ of $T$ divides the plane into two open regions exactly one of which is bounded. We refer to the bounded region as the {\em inside of $C$} and the unbounded region as the {\em outside of $C$}.  We say that $C$ is a {\em separating cycle} of $T$ if both regions contain vertices of $T$. If $S$ is a minimal vertex cut-set of $T$ then $S$ induces a separating cycle $C$.
It follows that every plane triangulation with at least four vertices is $3$-connected and that a plane triangulation with at least five vertices is 4-connected if and only if it contains no separating 3-cycles. Given an edge $e$ of $T$ which belongs to no separating 3-cycle of $T$, we can obtain a new plane triangulation $T/e$ by contracting the edge $e$ and its end-vertices to a single vertex (which is located at the same point as one of the two end-vertices of $e$), and replacing the multiple edges created by this contraction by single edges. 
For $X\subseteq V(T)$, let $T[X]$ be the subgraph of $T$ induced by $X$.
Given two vertices $u,v$ of $T$, we say that an edge $f\in E(T)$ is {\em $uv$-admissible} if $\{u,v\}$ is not contained in the unique face of $T-f$ of size four. Our motivation for considering such edges is that, if $f$ is $uv$-admissible and $T/f$ is $uv$-coincident rigid, then we can apply Lemma \ref{lem:split} to deduce that $T$ is $uv$-coincident rigid.

We can now give our sufficient condition for $uv$-coincident rigidity of plane triangulations in $\mathbb{R}^3$.
Its proof also  illustrates our strategy for proving Theorem~\ref{thm:coincident1}. 


\begin{theorem}\label{thm:coinplanar}
Let $T = (V, E)$ be a plane triangulation and $u, v \in V$. 
Suppose that $uv\notin E$ and that no separating $4$-cycle of $T$ contains both $u$ and $v$. 
Then $T$ is $uv$-coincident rigid.
\end{theorem}
\begin{proof}
Suppose the statement does not hold, and let $T=(V,E)$ be a counterexample such that $|V|$ is minimal. Then 
$T$ is not $uv$-coincident rigid for two non-adjacent vertices 
$u$ and $v$ which are not contained in any  separating $4$-cycle of $T$.
%
We first consider the case when $T$ is not 4-connected. 
Then $T$ has a separating 3-cycle $C$ and we have $T=T_1\cup T_2$ for two subtriangulations $T_1,T_2$ of $T$ with $T_1\cap T_2=C$. Note that $\{u,v\}\not\subseteq V(C)$ since $uv\not\in E$. If $u\in V(T_1)\sm V(C)$ and 
$v\in V(T_2)\sm V(C)$ then we can choose generic realisations $p_i$ of $T_i$ for $i=1,2$ with $p_1(w)=p_2(w)$ for $w\in V(C)$ and $p_1(u)=p_2(v)$. Then each $(T_i,p_i)$ is infinitesimally rigid and we can now use Lemma \ref{lem:glue} to deduce that the realisation $p$ of $T$ with $p(x)=p_i(x)$ for all $x\in V(T_i)$ is a $uv$-coincident infinitesimally rigid  realisation of $T$. Hence we may assume that $u,v\in V(T_1)$. Then we may apply the minimality of $T$ to deduce that there exists 
a generic $uv$-coincident infinitesimally rigid  realisation $p_1$ of $T_1$. Since 
$\{u,v\}\not\subseteq V(C)$, we may 
now choose a generic realisation  $p_2$ of $T_2$  with $p_2(w)=p_1(w)$ for all $w\in V(C)$ and proceed as in the previous subcase. 

It remains to consider the case when $T$ is 4-connected. Then $T/f$ is a plane triangulation for all $f\in E$.
Let $S$ be the (possibly empty) set of all edges of $T$ which lie on  a $uv$-path in $T$ of length two.
\begin{claim}\label{claim:coinplanar1}
$T$ has a $uv$-admissible edge $f$ with $f\notin S$.
\end{claim}
\begin{proof}
Recall that $u$ and $v$ are not adjacent in $T$.
If there is an edge $f$ incident to $u$ or $v$ and not contained in a $uv$-path of length two, then $f$ is $uv$-admissible with $f\notin S$.
Hence, we may assume that every edge  incident to $u$ or $v$ is contained in a $uv$-path of length two.
By 4-connectivity, there are at least four internally disjoint $uv$-paths of length two.
By planarity, at least one pair of these $uv$-paths forms a separating 4-cycle in $T$.
This contradicts the hypothesis that $\{u,v\}$ is not contained in a separating $4$-cycle.
\end{proof}

We also have the following.
\begin{claim}\label{claim:coinplanar2}
For every $uv$-admissible edge $f$ with $f\notin S$, $T$ has a separating 5-cycle  which contains $\{u,v\}$ and $f$.
\end{claim}
\begin{proof}
Let $f$ be a $uv$-admissible edge.
Let $w_1$ and $w_2$ be the vertices of $T$ which lie on the same face as $f$ and are not incident to $f$.
 Let $z$ be the new vertex we obtain when we contract $f$ to form $T/f$. We modify this notation when $f$ is incident to $u$ or $v$ by putting $z=u$ or $z=v$, respectively, to ensure that $\{u,v\}\subseteq V(T/f)$. Note that $uv\not\in E(T/f)$ since $f$ does not lie on a $uv$-path of length two.
%
If  $\{u,v,f\}$ is not contained in a separating 5-cycle of $T$,
then no separating $4$-cycle of $T/f$ contains both $u$ and $v$.  
So by  the minimality of $T$, there is a generic $uv$-coincident infinitesimally rigid framework $(T/f,p)$. 
Since $f$ is $uv$-admissible, $\{u,v\}\not\subseteq \{w_1,w_2,z\}$.  So the points in $p(w_1), p(w_2), p(z)$ are in general position and we may apply Lemma \ref{lem:split} to deduce that $T$ is $uv$-coincident rigid,
which is a contradiction.
\end{proof} 

By Claim~\ref{claim:coinplanar1} and Claim~\ref{claim:coinplanar2}, $T$ has a separating 5-cycle containing $\{u,v\}$.
Choose a separating 5-cycle $C$ of $T$ which contains $\{u,v\}$ and is such that some component $H$ of $T-V(C)$ has the minimum number of vertices. Let $C=uwvxyu$. By symmetry we may assume that $H$ is contained inside $C$. 
By the 4-connectivity of $T$ and the hypothesis that $T$ has  no separating 4-cycle containing $u$ and $v$, 
\begin{equation}\label{eq:coinplanar}
\text{$C$ has no chord incident to $u$ or $v$.}
\end{equation} 

Since $T$ is 4-connected, at least one of $x$ or $y$ is incident to a vertex inside 
$C$.
Without loss of generality, we assume $T$ has an edge $f=xz$ for some $z$ inside 
$C$.
Then $f$ is $uv$-admissible by (\ref{eq:coinplanar}) and clearly $f\notin S$.
Hence, by Claim~\ref{claim:coinplanar2}, $T$ has a separating 5-cycle $C'$ that contains $\{u,v\}$ and $f$.
By (\ref{eq:coinplanar}), $C'=uzxvz'u$ holds for some $z'\in V(T)$.
The minimum choice of $C$ and $H$ implies that $z'$ is not inside 
$C$.

\begin{claim}\label{claim:coinplanar3}
$N_T(w)=\{u,z,x,v,z'\}$.
\end{claim}
\begin{proof}
Since $T$ has the 5-cycle  $C'=uzxvz'u$, it suffices to show that $wz', wx, wz$ exist in $T$.
Observe first that $uwvz'$ forms a 4-cycle. Since there is no separating 4-cycle containing $u$ and $v$, 
we have $wz'\in E(T)$.

To see $wx, wz\in E(T)$, consider  the 5-cycle $C''=uzxvwu$.
The minimum choice of $C$ and $H$ implies that there are no vertices inside $C''$.
Moreover, by (\ref{eq:coinplanar}),  $wx$ and $wz$ are the only possible chords of $C''$.
This gives $wx, wz\in E(T)$.
\end{proof}

By Claim~\ref{claim:coinplanar3}, 
$C'$ and the set of edges incident to $w$ forms a wheel on six vertices, and the 4-connectivity of $T$ implies that $C'$ has no chord in $T$. 
Then $wx$ is $uv$-admissible with $wx\notin S$, and hence $T$ has a separating 5-cycle $C''$ which contains $\{u,v\}$ and $wx$ by Claim~\ref{claim:coinplanar2}. By (\ref{eq:coinplanar}), we have $C''=uwxvz''u$ for some vertex $z''\in V(T)\sm N_T(w)$. This  implies that $uwvz''u$ is a separating 4-cycle in $T$ which contains $\{u,v\}$, a contradiction.
\end{proof}

The plane triangulation $T$ obtained by joining two nonadjacent vertices $u,v$ to all vertices of a cycle $C$ of length at least four shows that the conclusion  of Theorem \ref{thm:coinplanar} may not hold if we remove the hypothesis that  no separating 4-cycle contains both $u$ and $v$: Lemma \ref{lem:coin_nec} implies that $T$ is not $uv$-coincident rigid since $T/uv$ has too few edges to be rigid.

\section{Contractible edges in plane triangulations}\label{subsec:contract}

Hama and Nakamoto \cite{HM}, see also Brinkmann et al \cite{BLSC}, showed that every 4-connected plane triangulation $T$ other than the octahedron has an edge $e$ such that $T/e$ is a 4-connected plane triangulation. 
For the analysis of braced triangulations, we will need  more detailed information on the distribution of such contractible edges. We will frequently use the following well known properties of a 4-connected plane triangulation $T$.
\begin{itemize}
\item $T/e$ is 4-connected if and only if $e$ belongs to no separating 4-cycle of $T$.
\item No separating 4-cycle in $T$ can have a chord.
\item No proper subgraph of $T$ can be a plane triangulation.
\item The octahedron is the unique 4-connected plane triangulation with at most six vertices.
\end{itemize}

Our first lemma is statement (b) in the proof of \cite[Theorem 0.1]{BLSC}. We include a proof for the sake of completeness.

\begin{lemma} \label{lem:contract}
    Let $T$ be a $4$-connected plane triangulation with at least seven vertices, $u$ be a vertex of $T$ of degree four and $e_1=uv_1,e_2=uv_2$ be two cofacial edges  of $T$.
    Then $T/e_i$ is $4$-connected for some $i=1, 2$.
\end{lemma}
\begin{proof}
Suppose, for a contradiction, that  $T/e_i$ is not 4-connected for both $i=1,2$.
 Let $C_1=v_1v_2v_3v_4v_1$ be the separating 4-cycle of $T$ which contains  the neighbours of $u$. 
Since  $T/e_1$ is not 4-connected, $T$ has a separating 4-cycle $C_2$ containing $e_1$. 
    Since no separating 4-cycle of $T$ can have a chord, 
    $C_2=wv_1uv_3w$ for some vertex $w \in V(T)\sm (V(C_1)\cup \{u\})$.
    Similarly, since  $T/e_2$ is not 4-connected, $T$ has a separating 4-cycle $ C_3=w'v_2uv_4w'$ for some $w' \in V(T)\sm (V(C_1)\cup \{u\})$. If $w' \neq w$
    then $T[V(C_1) \cup \{u,w,w'\}]$ contains a subgraph homeomorphic to $K_5$ contradicting 
    the planarity of $T$. On the other hand, if $w = w'$, then $T[V(C_1) \cup \{u,w\}]$ is a proper
    subtriangulation of $T$ and this contradicts the hypothesis that $T$ is a 4-connected triangulation.
\end{proof}

\begin{lemma} \label{lem:contract2}
    Suppose that $T$ is a $4$-connected plane triangulation with at least seven vertices and $F$ is 
    a  face of $T$. Then $T/e$ is $4$-connected for some edge $e$ of $T - V(F)$.
   \end{lemma}

\begin{proof} 
    Fix a plane embedding of $T$ with $F$ as the unbounded face. Since $T$ is 4-connected, some edge $e$ of $T$ is not incident with $V(F)$. If $T$ has no separating 4-cycle then $T/e$ would be 4-connected so we may assume that $T$ has a separating 4-cycle.
    Let $C = v_1v_2v_3v_4v_1$ be a separating 4-cycle such that the set of vertices inside $C$ is  minimal with respect to inclusion. Since $T$ is 4-connected, $C$ has no chords and hence, relabelling $V(C)$ if necessary,  we may assume that $v_1,v_2\not\in V(F)$.  Let $uv_1$ be an edge from a vertex $u$ inside $C$ to $v_1$. If $T/uv_1$  is 4-connected then we are done, so we may also assume that $v_1u$ belongs to a separating 4-cycle $C_1$ of $T$. The 
    minimality of $C$ implies that   $C_1=wv_1uv_3w$ for some vertex $w$ outside $C$, and hence that $u$ is the only vertex inside $C$ (otherwise $C_2=v_1uv_3v_2v_1$ or $C_3=v_1uv_3v_4v_1$ would contradict the minimality of $C$). This in turn implies that $u$ has degree 4 in $T$.  We can now use Lemma \ref{lem:contract} and the fact that $T/uv_1$ is not 4-connected to deduce that $T/uv_2$ is 4-connected.
\end{proof}

Let $C$ be a separating cycle in a plane triangulation $T=(V,E)$, $H_1$ and $H_2$ be the components of $G-C$ such that $H_1$ is inside $C$, and  $R\subset V$. We say that {\em $R$ crosses  $C$} if both  $H_1$ and $H_2$ contain a vertex in $R$. 
For $i=1,2$, let $H_i^+$ be the subgraph of $T$ consisting of $C\cup H_i$ and all edges of $T$ joining $C$ to $H_i$.
A separating 4-cycle $C$ is said to be  {\em inner minimal} if no other separating 4-cycle of $T$ is contained in  $H_1^+$.
 We say that $C$ is {\em outer minimal} if it satisfies the same condition with respect to $H_2^+$, and that $C$ is  {\em minimal} if it is either inner minimal or outer minimal.



\begin{lemma}\label{lem:contractedge2}
Let $T=(V,E)$ be a  4-connected plane triangulation, and $R\subset V$ with $4\leq |R|< |V|$.
Then  there exists an edge $f\in E$ such that $f$ is not induced by $R$  and  $R$ crosses every separating $3$-cycle of $T/f$.
\end{lemma}
\begin{proof} 
Let $f=xy$ be an edge of $G$ incident to a vertex $x\in V\sm R$. Then $T/f$ is a plane triangulation by the 4-connectivity of $T$. Furthermore, $R$ crosses every separating $3$-cycle of $T/f$ unless
\begin{enumerate}
\item[($\dagger$)]  $f$ belongs to a separating 4-cycle $C$ of $T$ and some component $H$ of $T-C$ contains no vertex of $R$.
\end{enumerate}

Suppose, for a contradiction, that the lemma is false
and that  $(T,R)$ is a counterexample.
Then, for every $x\in V\sm R$, each edge $f$ which is incident to $x$ belongs to a separating 4-cycle $C=x_1x_2x_3x_4x_1$ which satisfies condition ($\dagger$). 
We may assume that $x$, $f$ and $C$ have been chosen such that the specified component $H$ of $T-C$  has a minimal number of vertices. 
By symmetry, we  may also  assume that $H$ is contained inside $C$.  
Then $H^+$ is a plane graph and all interior faces of $H^+$ are triangles. 

Supoose $E(H)\neq\emptyset$, and let $f'=x'x''\in E(H)$.
Since $T$ is a counterexample, $f'$ belongs to a separating 4-cycle $C'$ of $T$ which satisfies ($\dagger$). The choice of $(x,f,C)$ now implies that $C'$ is not contained in $H^+$. This in turn implies that $C'\cap H^+$ is a path of length three joining two non-adjacent vertices, say $x_1,x_3$, of $C$, and $x_1x_3$ is a chord of $C$. This contradicts the hypothesis that  $T$ is a 4-connected plane triangulation.


It remains to consider the case when $E(H)=\emptyset$. Then $|V(H)|= 1$ and $H^+$ is a wheel on five vertices. Let $x_5$ be the unique vertex of $H$. Then $x_5\not\in R$, and each of the four triangular faces of $H^+$ incident to $x_5$ is a face of $T$.
If $T$ has more than six vertices, ($\dagger$) would be violated by Lemma~\ref{lem:contract}, which is a contradiction.
Hence $T$ has at most six vertices, and $T$ is the octahedron.
%
%
%
%
Let $x_6$ be the unique vertex of $T-H^+$. 
Then   the unique separating 4-cycle of $T$ which contains  $x_3x_5$, respectively $x_4x_5$, is $C_1=x_3x_5x_1x_6x_3$,  respectively $C_2=x_4x_5x_2x_6x_4$.
Since the edges $x_3x_5$ and $x_4x_5$ must both satisfy ($\dagger$),
each of $T-C$, $T-C_1$ and $T-C_2$ has a component with no vertices in $R$.
This gives $|V\sm R|\geq 3$ and contradicts the hypothesis that $|R|\geq 4$.
%
%
%
\end{proof}

\section{Braced triangulations}

%
Recall that a {\em braced plane triangulation}  is a graph $G=(V,E\cup B)$ which is the union of a plane triangulation $T=(V,E)$ and  a (possibly empty) set $B$ of additional edges which we refer to as {\em braces}. 
Given a braced plane triangulation $G=(T,B)$ and an edge $f$ of $T$ which belongs to no separating 3-cycle of $T$, we denote the braced plane triangulation obtained by contracting the edge $f$ by
$G/f=(T/f,B_f)$ where the set of bracing edges $B_f$ is obtained from $B$ by   replacing any multiple edges in $G/f$ by single edges (in particular any edge of $B$ which becomes parallel to an edge of $T/f$ is deleted). 

Our general strategy to prove Theorem~\ref{thm:coincident1} is to find a $uv$-admissible edge $f\in E(T)$ such that $G/f$ satisfies the hypotheses  of the theorem, apply induction to $G/f$, and then apply 
the vertex splitting lemma (Lemma~\ref{lem:split}) 
to return to $G$. 

\subsection{Admissible edges}

Let $T$ be a plane triangulation.
Observe that, when $e=uv\in E(T)$,  $f \in E(T)$ is $uv$-admissible if and only if $e,f$ do not belong to a common face of $T$.
Our first result shows how Lemma~\ref{lem:split} can be used to complete the inductive step in our proof of Theorem~\ref{thm:coincident1}.

.
\begin{lemma}\label{lem:split_uv}
Let $G=(T, B)$ be a braced triangulation, and $e=uv\in E(T)$.
Suppose $f\in E(T)$ is $uv$-admissible and $G/f$ is $uv$-coincident rigid in $\mathbb{R}^3$.
Then $G$ is $uv$-coincident rigid.
\end{lemma}
\begin{proof}
Let $(G/f,p)$ be a generic $uv$-coincident rigid realization of $G/f$.
Let $f=ab$.
Let $w$ be the new vertex of $G/f$ created by contracting $f$, and $B_f$ be the set of bracing edges in $G/f$.
We may assume $B_f$ is a subset of $B$. (If an edge $e'$ in $B_f$ corresponds to more than one edge in $B$, then we choose one of these edges arbitrarily and identify it with $e'$.)
We may obtain a spanning subgraph $H$ of $G$  
from $G/f$ such that $N_H(a)\cap N_H(b)=N_T(a)\cap N_T(b)$ by applying the vertex splitting operation at $w$.
Since $p$ is a generic $uv$-coincident realization and $f$ is $uv$-admissible, 
 $\{p(x): x\in \{a\}\cup ( N_H(a)\cap N_H(b))\}$ is in general position.
Hence, the $uv$-coincident rigidity of $G$ follows from  Lemma~\ref{lem:split}. 
\end{proof}


 
We next give a lemma which will be used to solve the most difficult case in our proof of Theorem~\ref{thm:coincident1}: when $G$ is obtained from a 4-connected plane triangulation $T$ by adding exactly one brace $xy$.
Specifically, for 
$e=uv\in E(T)$, and $xy\notin E(T)$,  
we need to find a $uv$-admissible edge $f\in E(T)$ such that
\begin{itemize}
\item[(i)] $(T+xy)/f$ is a braced planar triangulation (with at least one brace),
\item[(ii)] $(T+xy)/f$ is 4-connected, and 
\item[(iii)] $e$ does not belong to any separating 3-cycle of $T/f$. 
\end{itemize} 
This can be guaranteed if $f$ satisfies the following conditions in $T$: 
\begin{itemize}
\item[(i')] $f$ does not belong to an $xy$-path of length two in $T$,
\item[(ii')] for every separating 4-cycle $C$ of $T$ which contains $f$, $\{x,y\}$ crosses $C$, and 
\item[(iii')] for every separating 4-cycle $C$ of $T$ which contains $f$, $e\notin E(C)$.
\end{itemize}
The following lemma shows that we can always find a $uv$-admissible edge $f\in E(T)$ satisfying (i'), (ii'), and (iii').

\begin{lemma}\label{lem:contractedge1}
Let $T=(V,E)$ be a  4-connected plane triangulation on at least seven vertices, $e=uv\in E$  and $x,y$ be two non-adjacent vertices of $T$. 
Then  there exists a $uv$-admissible edge $f\in E$ such that  
$f$ does not belong to an $xy$-path of length two and, for every separating $4$-cycle $C$ of $T$ which contains $f$,  we have $e\not\in E(C)$ and $\{x,y\}$ crosses  $C$.
\end{lemma}

%
\begin{proof}  Suppose, for a contradiction,  that the lemma is false and $T$ is a counterexample. Let $F,F'$ be the  faces of $T$ which contain $e$ and let $w$, respectively $w'$, be the vertex of $F$, respectively $F'$, that is different from $u,v$. 
Let
$S$ be the (possibly empty) set of all edges of $T$ which lie on an $xy$-path in $T$ of length two. The assumption that $T$ is a counterexample implies that
\begin{enumerate}
\item[($*$)] every edge of $E(T)\setminus (E(F)\cup E(F')\cup S)$ belongs to a separating 4-cycle $C$ of $T$ such that either $e\in E(C)$ or $C$ is not crossed by $\{x,y\}$.
\end{enumerate}
Since $T$ is a triangulation with more than five vertices, $T-(E(F)\cup E(F'))$ contains 
        a triangle. Since $S$ induces a bipartite subgraph, this implies that $E(T) \sm (E(F) \cup E(F') \cup S)$ is nonempty, so $T$ has a separating 4-cycle by ($*$).

\begin{claim}
        \label{clm:minsepcycle}
        Let $C$ be a minimal separating 4-cycle of $T$ and 
       $H$ be a component of $T-C$ such that no separating $4$-cycle of $T$ is contained in $H^+$. 
       Suppose that $e\not\in E(H^+)\sm E(C)$. Then $H^+$ is a  wheel on five vertices. Furthermore:\\[1mm]
        (a) if $e\in E(C)$,
        then    the unique vertex $z$ of $H$ satisfies $z\in \{w,w'\}\cap \{x,y\}$;\\[1mm]
        (b) if $e\not\in E(C)$ then  $\{x,y\}\subset V(C)$, exactly one of $\{u,v\}$, say $u$, belongs to $V(C)\sm\{x,y\}$ and $v$ is adjacent to the unique vertex of $V(C)\sm \{u,x,y\}$.
           \end{claim}
\begin{proof}[Proof of Claim]
        By symmetry, it will suffice to prove the claim when $C$ is inner minimal. The 4-connectivity of $T$ and minimality  of $C$ imply that $H^+$ is either a wheel on five vertices or is 4-connected. 
        
Suppose $H^+$ is 4-connected. 
We first show that we can choose  an edge $f\in E(H^+)\setminus E(C)$ such that $f\notin E(F)\cup E(F')\cup S$.
If there exists  a triangular face $F_1$ of $H^+$ such that $E(F_1)\subseteq E(H^+)\sm (E(C)\cup E(F)\cup E(F'))$,
then we can choose an edge    $f\in E(F_1)\sm S$ since $S$ induces a  bipartite subgraph of $T$.   
This edge $f$ satisfies $f\notin E(F)\cup E(F')\cup S$ as required.
On the other hand, if $H^+$ has no triangular face $F_1$ such that $E(F_1)\subseteq E(H^+)\sm (E(C)\cup E(F)\cup E(F'))$, then, since  $e\not\in E(H^+)\sm E(C)$,  $e\in E(C)$ holds and $H^+-e$ forms the wheel on six vertices.
In this case $H^+$ would have a vertex of degree three, contradicting the 4-connectivity of $H^+$.

Thus, we can choose an edge $f\in E(H^+)\setminus E(C)$ such that $f\notin E(F)\cup E(F')\cup S$,
and $f$  is contained in a separating 4-cycle $C'$ of $T$ by ($*$). 
The minimality of $C$  implies that $C'\not\subseteq H^+$ and the fact that $|V(H)|\geq 2$ 
now implies that either $C'$ or $C$ has a chord. This contradicts the 4-connectivity of $T$.
Hence $H^+$ is a wheel on five vertices. 

\medskip

(a) Suppose $e\in E(C)$. Then exactly one of  $w,w'$ is contained in $H^+$. Relabelling if necessary, we may assume that $V(H)=\{w'\}$. Let $C=uvstu$. By Lemma \ref{lem:contract}, at least one of the cofacial edges $w's,w't$, say $w's$, does not belong to a separating 4-cycle of $T$. We can now apply ($*$) to deduce that $w's\in S$ so $\{w',s\}\cap\{x,y\}\neq \emptyset$. 
If $s\in \{x,y\}$, say $s=x$, then the facts that $w's\in S$ and $xy\not\in E$ give $y=u$. 
Then $w't$ satisfies $w't\in E(T)\setminus (E(F)\cup E(F')\cup S)$.
Moreover, since every separating 4-cycle $C'$ which contains $w't$ is of the form $tw'vu't$ for some $u'\in V$, 
$C'$ does not contain $uv$ and is crossed by $\{x,y\}$.
This contradicts  ($*$).

\medskip

(b) Suppose $e\not\in E(C)$. Since $e\notin E(H^+)\setminus E(C)$, we have $e\notin E(H^+)$.
Let $C=v_1v_2v_3v_4v_1$. 
%
%
%
Since $H^+$ is a wheel on five vertices, the unique vertex $z$ of $H$ has degree four in $T$ and we can apply Lemma \ref{lem:contract} to $T$ to deduce that, after a possible relabelling of $V(C)$,  neither $zv_1$ nor $zv_3$ belong to a separating 4-cycle of $T$. 
By ($*$)  and $e\notin E(H^+)$, we have  $\{zv_1,zv_3\}\subseteq  S$ and $\{x,y\}=\{v_1,v_3\}$. Then $zv_2,zv_4\not\in S$ and ($*$) implies that there exists a separating 4-cycle $C'=v_2zv_4z'v_2$ of $T$ such that either $e\in E(C')$ or $\{x,y\}$ does not cross $C'$. Since  $\{x,y\}$ crosses $C'$, we have $e\in E(C')$. Hence $z'\in \{u,v\}$, $|\{u,v\}\cap \{v_2,v_4\}|=1$ and (b) holds. 
    \end{proof}

\begin{claim}
        \label{clm:sepcycle}
  Some separating 4-cycle of $T$ does not contain $e$.
           \end{claim}
\begin{proof}[Proof of Claim]
Suppose for a contradiction that every separating 4-cycle of $T$ contains $e$. 
Choose two separating 4-cycles $C_0 = uvu_1u_2$ and $C'_0 = uvu'_1u'_2$, such that  $C_0$ is outer minimal and $C_0'$ is inner minimal. Relabelling $w,w'$ and $x,y$ if necessary, Claim \ref{clm:minsepcycle} implies that $w=x$ is the unique vertex outside $C_0$ and $w'=y$ is the unique vertex inside $C_0'$. If $xu_1\not\in S$ then $(*)$ would give us a separating 4-cycle  $C'$ which contains $xu_1$. 
Then $C'$ cannot contain $e$ 
(otherwise, since $C'$ is separating, $C'=xu_1uvx$  holds and $C_0$ has the chord $uu_1$, contradicting the 4-connectivity of $T$)  so $C'$ would be  the required separating 4-cycle in $T$.  Hence  $xu_1\in S$. By symmetry, we also have $xu_2,yu_1'yu_2'\in S$. This implies that
$C_0 = C'_0$ and $\{x,y\} \cup V(C_0)$ induces a copy of the octahedron
in $T$. This contradicts the assumption that $T$ is 4-connected and has at least seven
 vertices.
 \end{proof}
  
We can now complete the proof of the lemma. By Claim \ref{clm:sepcycle}, $T$ has a separating 4-cycle $C$ which does not contain $e$. By symmetry we may assume that $e$ lies outside of $C$. Replacing $C$ by a minimal separating 4-cycle which is contained inside $C$ if necessary, we may assume that $C$ is inner minimal.
Relabelling $u,v$ if necessary, Claim \ref{clm:minsepcycle} now implies that $C=xuytx$ for some vertex $t\neq v$ which is adjacent to $v$ and there is exactly one vertex $z$ inside $C$.

Consider the separating 4-cycle $C'=uvtzu$ of $T$.  By symmetry we may assume that  $x,w$ lie outside $C'$ and $y,w'$ lie inside  $C'$. 
If $x=w$ and $y=w'$ then $T$ would contain a copy of the octahedron. This would contradict the hypothesis that $T$ is 4-connected plane triangulation with at least seven vertices so we may also assume by symmetry that $x\neq w$. Then $C''=uxtvu$ is another separating 4-cycle of $T$. We may now choose an outer minimal separating 4-cycle $C_1$ of $T$
which is contained in the closed region outside of $C''$. Let $H_1$ be the 
component of $T-C_1$ outside $C_1$. 

We will show that the pair $(C_1,H_1^+)$ contradicts Claim \ref{clm:minsepcycle}.  We have
$e\not\in E(H_1^+)\sm E(C_1)$ since $e\in E(C'')$ so 
$(C_1,H_1^+)$ satisfies the hypotheses of Claim \ref{clm:minsepcycle}. 
We also have  $y\not\in V(H_1^+)$ since $y$ is inside $C''$ and $x\not\in V(H_1^+)\sm V(C_1)$ since $x\in V(C'')$. Hence
$(V(H_1^+)\sm V(C_1))\cap \{x,y\}=\emptyset$ so part (a) of Claim \ref{clm:minsepcycle} does not hold, and 
$y\not\in V(C_1)$ so part (b) of Claim \ref{clm:minsepcycle} does not hold.
\end{proof}

\subsection{Coincident rigidity: proof of Theorem~\ref{thm:coincident1}}

We will  use the well known result that the maximal 4-connected subgraphs,   or {\em $4$-blocks}, of a plane triangulation $T$ form a tree like structure. More precisely we can define the {\em $4$-block tree}  of $T$ to be the 
tree $\cT$ whose vertex set  is the union of the set of 4-blocks and the set of separating 3-cycles of $T$,  in which a 4-block $D$ is adjacent  to  a separating 3-cycle $C$  if $C\subseteq D$.

\paragraph{Proof of Theorem \ref{thm:coincident1}.}
We proceed by contradiction. Suppose the theorem is false and let  $G=(T,B)$ be a counterexample such that $|E(G)|$ is as small as possible. 
%
If $|V|=5$ then $G\cong K_5$ and the theorem holds since $K_5$  is  $uv$-coincident rigid for every edge $uv\in E(K_5)$. Hence we may assume that $|V|\geq 6$.  
Consider the following two cases.

\paragraph{\boldmath Case 1: $T$ is 4-connected.} Choose $b=xy\in B$.
If $T+b\neq G$ then we can apply induction to $G-b$. Hence we have $G=T+b$. Let $F$ and $F'$ be the faces of $T$ which contain $e=uv$ and let $S$ be the set of edges of $T$ which lie on an $xy$-path of length two.
Since $T$ is 4-connected, we have $|V|\geq 6$ with equality only if $T$ is the octahedron.

Suppose $T$ is the octahedron. Then $T-V(F\cup F')\cong K_2$. Let $f$ be the unique edge in $T-V(F\cup F')$. 
Then $f$ is $uv$-admissible.
 If $f\in S$ then $b$ is incident with an end-vertex of both $e$ and $f$ and, up to symmetry, there is a unique choice for $e$ and $b$. Taking $e=uv$, $f=wz$ and $b=uz$ we have $G-w+g\cong K_5$ for a unique edge $g$ joining two neighbours of $w$ and we can now use Lemma \ref{lem:1ext} to obtain an infinitesimally rigid  
$uv$-coincident realisation of $G$ from an infinitesimally rigid  
$uv$-coincident realisation of $K_5$. 
Hence $f\not\in S$. Then $G/f\cong K_5$ and 
$G/f$ is $uv$-coincident rigid.
Since $f$ is $uv$-admissible, Lemma \ref{lem:split_uv} now implies that $G$ is $uv$-coincident rigid.

Hence  $|V(T)|\geq 7$. 
Lemma \ref{lem:contractedge1} now implies that there exists an edge $f\in E\sm (E(F)\cup E(F')\cup S)$ such that $\{x,y\}$ crosses every separating 3-cycle of $T/f$ that contains $e$. Then $G/f=T/f+b$ is a 4-connected braced triangulation and $b$ is a brace of $G/f$. Hence  $G/f$  is  $uv$-coincident rigid by the minimality of $G$.
Since $f$ is $uv$-admissible, $G$  is  $uv$-coincident rigid by Lemma~\ref{lem:split_uv}. 

\paragraph{\boldmath Case 2: $T$ is not 4-connected.}
Since $e$ does not belong to any separating $3$-cycle of $T$, $e$ belongs to a unique 4-block $D_e$ of $T$.
Since $T+B$ is 4-connected we can choose a path $P=D_0X_1D_1\ldots X_kD_k$ in the 4-block tree $\cT$ of $T$ such that $D_i=D_e$ for some $0\leq i\leq k$ and  some brace $b=xy\in B$ has $x\in V(D_0)\sm X_1$ and $y\in V(D_k)\sm X_k$. Let  $T'=\bigcup_{i=0}^k D_i$.  If $T'+b\neq G$ then the minimality of  $|E(G)|$ implies that $T'+b$ is $uv$-coincident rigid and we can now use 
Lemma \ref{lem:glue} to deduce that $G$ is $uv$-coincident rigid.
Hence $T'+b= G$. By symmetry, we may assume that $D_e\neq D_k$. 

We claim that $D_k$ is isomorphic to $K_4$.
To see this, suppose $D_k\not\cong K_4$. Then $D_k$ is 4-connected. 
Let $R=X_k+y$.
Then Lemma \ref{lem:contractedge2} gives an edge $f\in E(D_k)$ such that $f$ is 
not induced by $R$ and $R$ crosses every separating 3-cycle of $D_k/f$. Then $(T/f)+b$ is  4-connected so the minimality of $|E(G)|$ implies that $T/f+b$  is  $uv$-coincident rigid. Since $e\not\in E(D_k)$ , $f$ is $uv$-admissible and  Lemma~\ref{lem:split_uv}  implies that 
$G$  is  $uv$-coincident rigid, a contradiction. 

Hence $D_k$ is the complete graph on $X_k\cup \{y\}$. 
Denote $X_k=\{a_1, a_2, a_3\}$.
We next show that there is an edge $f'=xa_i$ such that $T-y+f'$ is 4-connected for some $i\in \{1,2,3\}$. 
If $a_i\in X_k\setminus X_{k-1}$, then $T-y+f'$ is 4-connected  for $f'=xa_i$.
Hence, assume $k=1$ and $X_1=\{a_1,a_2,a_3\}$. 
Then $T-y=D_0$, which is a 4-connected triangulation. 
Also, $xa_i\not\in E(D_0)$ for some $i$  since otherwise  $D_0\cong K_4$ and $|V|=5$, contradicting $|V|\geq 6$. 

Thus, without loss of generality, we have that $T-y+xa_1$ is a 4-connected braced plane triangulation.
Since $e$ does not belong to any separating 3-cycle in $T-y$, by induction $T-y+xa_1$ is $uv$-coincident rigid. 
Observe further that $G$ can be obtained from $T-y+xa_1$ by applying the vertex splitting operation to $a_1$.
Since $e$ is not induced by $\{a_1, a_2, a_3\}$ (as $e$ belongs to no separating 3-cycle in $T$), 
we can apply Lemma~\ref{lem:split} to conclude that $G$ is $uv$-coincident rigid.
\qed

\section{Global rigidity: proof of Theorem~\ref{JKT}}
In this section, we use  Theorems~\ref{vsplitc} and~\ref{thm:coincident1} to prove Theorem~\ref{JKT}.
 
\paragraph{Proof of Theorem \ref{JKT}.}
Let $G=(T, B)$ be a braced triangulation with at least five vertices, where $T=(V,E)$ is a plane triangulation and $B$ is a set of braces.
Necessity follows from (the 3-dimensional version of) a result of Hendrickson  \cite[Theorem 5.9]{H} which implies that every globally rigid graph on at least five vertices 
is 4-connected and 
remains rigid after the removal of any edge. Hence, if $G$ is globally rigid, then $G$ is 4-connected, and $B\neq\emptyset$ since otherwise  $G=T$ would have $3|V|-6$ edges so $G-e$ would have too few edges to be rigid for all $e\in E(G)$. 

To prove sufficiency we assume  $G$ is 4-connected and $B\neq \emptyset$. We prove that $G$ is globally rigid in $\R^3$ by induction on $|V|$. 
      If $|V|=5$ then $G\cong K_5$ and we are done since $K_5$ is globally rigid. Hence we may suppose that $|V|\geq 6$. 

We first consider the case when $T$ is a 4-connected plane triangulation. 
Choose a brace $xy\in B$ and let $S$ be the set of edges of $T$ which lie on an $xy$-path of length two. 
If  $|V(T)|=6$, then $T$ is the octahedron and  
 $G/f\cong K_5$ for all $f\in E(T)\sm S$, so $G/f$ is globally rigid. This would imply that  $G$ is globally rigid
by Theorems  \ref{thm:coincident1} and  \ref{vsplitc}.
 Hence we may assume that $|V(T)|\geq 7$. We can now apply Lemma \ref{lem:contractedge1} (taking $e$ to be an arbitrary edge of $T$) to deduce  that there exists an edge $f\in E(T)\setminus S$ such that $T/f$ is 4-connected.
Then $T/f+xy$ is a 4-connected braced triangulation so is globally rigid by induction, and we can again use Theorems  \ref{thm:coincident1} and  \ref{vsplitc} to deduce that $G$ is globally rigid.

Hence we may assume that $T$ is not 4-connected. Choose a fixed embedding of $T$ in the plane and let $C$ be a separating 3-cycle in $T$ such that the component $H$ of $G-C$ which lies inside $C$ is minimal with respect to inclusion. Let $H^+$ be the subtriangulation of $T$ obtained from $C\cup H$ by adding all edges of $T$ from $H$ to $C$.      Since $G$ is 4-connected there
    is a brace  $xy \in B$ with $x \in V(H)$ and $y \in V(T) \sm V(H^+)$. 
The minimality 
    of $H$ implies that $H^+$ is 4-connected or is isomorphic to $K_4$.

\paragraph{\boldmath Case 1: $H^+$ is isomorphic to $K_4$.} We first consider the subcase when there exists a vertex $z\in V(C)$ which is not adjacent to $y$ in $T$.
Then $G/xz$ is a 4-connected braced triangulation with at least one brace so is globally rigid by induction. In addition, $T-x$ is a plane triangulation so is rigid by Gluck's Theorem. This allows us to construct an $xz$-coincident infinitesimally rigid realisation $(G,p)$ from a generic infinitesimally rigid realisation $(G-x,p')$ by putting $p(x)=p'(z)$ and using the fact that $x$ has  at least three neighbours other than $z$ in $G$. Theorem \ref{vsplitc} now implies that $G$ is globally rigid. 

It remains to consider the subcase when, for every brace $xy$ incident to $x$ in $G$, $y$ is adjacent to every vertex of $C$ in $T$. Planarity now implies that $xy$ is the unique brace incident to $x$ and
$V(C)\cup \{y\}$ induces a copy of $K_4$ in $T$. The fact that $|V(T)|\geq 6$ now implies that $T-x$ is not 4-connected. In addition, 
$G-x=(T-x,B-xy)$ is a 4-connected braced plane triangulation, and has at least one brace since $T-x$ is not 4-connected. Then $G-x$ is globally rigid, by induction, and the fact that $x$ has degree four in $G$ now implies that $G$ is globally rigid.

\paragraph{\boldmath Case 2: $H^+$ is 4-connected.} Then $|V(H)|\geq 2$ and the minimality of $H$ now implies that some vertex $z\in V(C)$ is not adjacent to $x$ in $H^+$. Then $G'=H^++xz$ is a braced 4-connected  plane triangulation with exactly one brace. By Theorem \ref{thm:coincident1}, $G'$ has an infinitesimally rigid $uv$-coincident realisation for all edges $uv$ of $H^+$. We can now use Lemma \ref{lem:elem}, taking $G_1:=G'$ and $G_2:=G-V(H)$, to deduce:
\begin{enumerate}
\item[($\diamond$)] $G$ has an infinitesimally rigid $uv$-coincident realisation for all
edges $uv$ of $H^+$ with $\{u,v\}\not\subset V(C)\cup\{x\}$.
\end{enumerate}

Suppose $H^+$ is isomorphic to the octahedron. Let $uv$ be the unique edge of $H^+$ which is not incident to a vertex in $V(C)\cup \{x\}$. Then $G/uv=T/uv+xy$ is a 4-connected braced triangulation with at least one brace so is globally rigid by induction. We can now use Theorem \ref{vsplitc} and ($\diamond$)   to deduce that $G$ is globally rigid.

It remains to consider the subcase when  $|V(H^+)| \geq 7$. By Lemma \ref{lem:contract2}, there is an edge
    $uv \in E(H)$ such that $H^+/uv$ is 4-connected. Then
    $G/uv$ is a 4-connected braced triangulation with at least one brace which, by induction, is globally rigid. Theorem \ref{vsplitc} and ($\diamond$) now imply  that $G$ is globally rigid.
\hfill\qed

\paragraph{Acknowledgements} We would like to thank the referees for their careful reading and helpful comments which have greatly improved this paper. 
Our work was supported by 
JST ERATO Grant Number JPMJER1903,  
JSPS KAKENHI Grant Number 18K11155, and EPSRC overseas travel grant EP/T030461/1.

\end{document}